\documentclass[a4paper,12pt]{amsart}

\usepackage[dvips]{graphicx}
\usepackage[all]{xy}
\usepackage{amscd}
\usepackage{amsmath,amssymb}
\title[Degenerate center singularities]{Symmetries of degenerate center singularities of plane vector fields}
\author{Sergiy Maksymenko}
\address{Topology dept., Institute of Mathematics of NAS of Ukraine, Tere\-shchenkivska st. 3, Kyiv, 01601 Ukraine}
\input{macro.tex}

\newcommand\flatsign[1]{\bar{#1}}

\newcommand\TC{TC}
\newcommand\hdif{h}
\newcommand\kdif{k}
\newcommand\orb{o}
\newcommand\dd[1]{\frac{\partial}{\partial #1}}
\newcommand\ddd[2]{\frac{\partial #1}{\partial #2}}
\newcommand\Wr[1]{\mathsf{W}^{#1}}

\newcommand\dAx{\flatsign{X}}
\newcommand\dAy{\flatsign{Y}}

\newcommand\DD{\mathcal{D}}
\newcommand\EE{\mathcal{E}}

\newcommand\LA{L(\AFld)}

\newcommand\tafunc{\widetilde{\afunc}}
\newcommand\thdif{\widetilde{\hdif}}

\newcommand\tdisk{B}
\newcommand\tdisko{\check{\tdisk}}

\newcommand\torb{\widetilde{\orb}}
\newcommand\rrho{r}

\newcommand\tbfunc{\widetilde{\bfunc}}

\newcommand\tz{\widetilde{z}}
\newcommand\jt[1]{\jet(#1)}

\newcommand\Apl{\mathcal{A}_{++}}
\newcommand\Amn{\mathcal{A}_{--}}
\newcommand\Aopl{\mathcal{A}_{+-}}
\newcommand\Aomn{\mathcal{A}_{-+}}
\newcommand\Ag{\mathcal{A}}

\newcommand\Ahom{\mathbf{A}}
\newcommand\Bhom{\mathbf{B}}

\begin{document}
\begin{abstract}
Let $D^2 \subset\mathbb{R}^2$ be a closed unit $2$-disk centered at the origin $O\in \mathbb{R}^2$, and $F$ be a smooth vector field such that $O$ is a unique singular point of $F$ and all other orbits of $F$ are simple closed curves wrapping once around $O$.
Thus topologically $O$ is a ``center'' singularity.
Let $\mathcal{D}^{+}(F)$ be the group of all diffeomorphisms of $D^2$ which preserve orientation and orbits of $F$.

In the previous paper the author described the homotopy type of $\mathcal{D}^{+}(F)$ under assumption that the $1$-jet $j^1 F(O)$ of $F$ at $O$ is non-degenerate.
In this paper degenerate case of $j^1 F(O)$ is considered.
Under additional ``non-degeneracy assumptions'' on $F$ the path components of $\mathcal{D}^{+}(F)$ with respect to distinct weak topologies are described.
\end{abstract}

\maketitle 
\section{Introduction}
Let $\disk=\{ x^2+y^2\leq 1\} \subset\RRR^2$ be a closed unit $2$-disk centered at the origin $\orig\in\RRR^2$, $\Vman \subset\RRR^2$ be a closed subset diffeomorphic with $\disk$, $z\in\Int\Vman$, and $$\AFld=\AFld_1\dd{x}+\AFld_2\dd{y}$$ be a $\Cinf$ vector field on $\Vman$.
We will say that $\AFld$ is a \myemph{\TC\ vector field on $\Vman$ with topological center at $z$} if it satisfies the following conditions:
\begin{enumerate}
\item[(T1)] $z$ is a unique singular point of $\AFld$,
\item[(T2)] $\AFld$ is tangent to $\partial\Vman$, so $\partial\Vman$ is an orbit of $\AFld$, and
\item[(T3)] all other orbits of $\AFld$ are closed.
\end{enumerate}
If $\Vman=\disk$, then we will always assume that $\AFld(\orig)=0$.

Let $\AFld$ be a \TC\ vector field on $\disk$.
Then it easily follows from Poincar\'e-Bendixson theorem,~\cite{PalisdeMelo}, that there exists a \myemph{homeomorphism} $$\hdif=(\hdif_1,\hdif_2):\disk\to\disk$$ such that $\hdif(\orig)=\orig$ and for every other orbit $\orb$ of $\AFld$ its image $\hdif(\orb)$ is the circle of some radius $c\in(0,1]$ centered at origin, see Figure~\ref{fig:tcvf}.
\begin{figure}[ht]
\includegraphics[height=1.5cm]{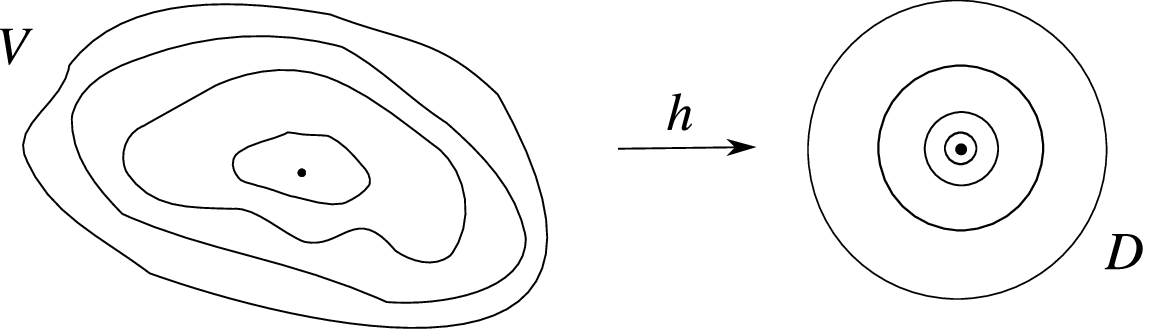}
\caption{}
\protect\label{fig:tcvf}
\end{figure}
This motivates the term \TC\ which we use, see~\cite{Maks:LocExtrPer}.
Moreover, it can be assumed that the restriction $\hdif:\disko\to\disko$ is a $\Cinf$ diffeomorphism.

\begin{definition}\label{def:first_strong_int}
A continuous function $\func:\disk\to[0,1]$ will be called a \myemph{first strong intergal} for $\AFld$ if 
\begin{enumerate}
 \item[(i)]
$\func$ is $\Cinf$ on $\disko$ and has no critical points in $\disko$,
 \item[(ii)] $\func^{-1}(0)=\orig$, $\func^{-1}(1)=\partial\disk$, and for $c\in [0,1]$ the set $\func^{-1}(c)$ is an orbit of $\AFld$.
\end{enumerate}
\end{definition}
Notice that we do not require that $\func$ is $\Cinf$ at $\orig$.
It also follows from (ii) that $\func$ takes distinct values on distinct orbits.
For instance if $\hdif$ is $\Cinf$ on $\disko$, then the function 
$\func =\hdif_1^2+\hdif_2^2$ is the \myemph{first strong integral} for $\AFld$.

Consider the following matrix
$$
\nabla\AFld = \left(
\begin{array}{cc}
 \ddd{\AFld_1}{x}(\orig) & \ddd{\AFld_1}{y}(\orig) \\ [2mm]
 \ddd{\AFld_2}{x}(\orig) & \ddd{\AFld_2}{y}(\orig)
\end{array}
\right).
$$
We will call $\nabla\AFld$ a \myemph{linear part} or \myemph{linearization} of $\AFld$ at $\orig$.

Suppose $\AFld$ is a \TC\ vector field and $\nabla\AFld$ is non-degenerate.
Then it can be shown that there are local coordinates at $\orig$ in which $$j^1\AFld(\orig)=-ay\dd{x}+ay\dd{y}$$ for some $a\not=0$, so the $1$-jet of $\AFld$ at $\orig$ is a ``rotation''.
This class of singularities is well-studied from many points of view, see e.g. \cite{Poincare, Lyapunov, Takens:AIF:1973, Belitsky:FA:85, Sibirsky1, Sibirsky2}.
In particular, in~\cite{Takens:AIF:1973} normal forms of such vector fields are obtained.

Denote by $\DA$ the group of $\Cinf$ diffeomorphisms $\hdif$ of $D^2$ such that $\hdif(\orb)=\orb$ for each orbit $\orb$ of $\AFld$.
Let also $\DApl$ be the subgroup of $\DA$ consisting of all orientation preserving diffeomorphisms, and $\DAd$ be the subgroup of $\DApl$ consisting of all diffeomorphisms fixed of $\partial\disk$.
We endow $\DA$, $\DApl$, and $\DAd$ with the weak $\Wr{\infty}$ topologies, see \S\ref{sect:Kr_deform}.

The main result of~\cite{Maks:LocExtrPer} describes the homotopy types of $\DAd$ and $\DApl$ for a \TC\ vector field with non-degenerate $\nabla\AFld$, see (i) of Proposition~\ref{pr:info_on_TC}.
Moreover, in~\cite{Maks:PartitPresDif} 

In this paper we study \TC\ vector fields with degenerate $\nabla\AFld$ and describe path components of $\DApl$ and $\DAd$ under additional ``non-degeneracy'' assumptions, see Theorems~\ref{th:DA_EA_weak_he}, \ref{th:EAd_EApl}, and \ref{th:EA_non-deg-cond}.
The obtained results are not so complete as in the non-degenerate case due to the variety of normal forms.

Together with results of~\cite{Maks:LocExtrPer} these theorems will be used in another paper to extend~\cite{Maks:AGAG:2006} to a large class of $\Cinf$ functions on surfaces.

\section{Formulation of results}
Let $\AFld$ be a \TC\ vector field on $\disk$.
Denote by $\EA$ the subset of $\Ci{\disk}{\disk}$ which consists of all maps $\hdif:\disk\to\disk$ satisfying the following conditions:
\begin{itemize}
 \item 
$\hdif(\orb)=\orb$ for every orbit $\orb$ of $\AFld$.
In particular, $\hdif(\orig)=\orig$.

 \item
$\hdif$ is a local diffeomorphism at $\orig$.
\end{itemize}

Evidently, $\EA$ is a subsemigroup of $\Ci{\disk}{\disk}$ with respect to the usual composition of maps.
Consider the map
$$
\jet:\EA\to \GLR{2}, \qquad \jet(\hdif) = J(\hdif,\orig),
$$
associating to every $\hdif\in\EA$ its Jacobi matrix $J(\hdif,\orig)$ at $\orig$.
Let $$\LA := \jet(\EA)$$
be the image of $\jet$.
Then a priori $\LA$ is a subsemigroup of $\GLR{2}$.

Let $\EApl=\jetinv(\GLRpl{2})$ be the subset of $\EA$ consisting of all maps $\hdif$ with positive Jacobian at $\orig$.

Let also $\EAd \subset \EApl$ be the subsemigroup consisting of all maps $\hdif$ fixed on $\partial\disk$, i.e. $\hdif(x)=x$ for all $x\in\partial\disk$.
Evidently,
\begin{equation}\label{equ:inclusions:D_E}
\DA\subset\EA,
\qquad 
\DApl\subset\EApl,
\qquad
\DAd\subset\EAd.
\end{equation}

For $r=0,1,\ldots,\infty$ denote by $\EAr{r}$, ($\EAplr{r}$, etc.) the space $\EA$, ($\EApl$, etc.) endowed with weak $\Wr{r}$ topology~\S\ref{sect:Kr_deform}.
Let also $\EidAr{r}$, ($\EidAplr{r}$, etc.) be the path component of the identity map $\id_{\disk}$ in $\EA$, ($\EApl$, etc.) with respect to $\Wr{r}$ topology at least for $r\geq1$.
Evidently, each $\hdif\in\EA\setminus\EApl$ (if it exists) changes orientation of $\disk$, whence $\EAplr{r}$ consists of full path components of $\EAr{r}$ for any $r$.
In particular,
$$ \EidAplr{r} = \EidAr{r} \qquad \forall r=0,1,\ldots,\infty.$$

It turns out that it is more convenient to work with $\EA$ instead of $\DA$.
Moreover, the following theorem shows that such a replacement does not loose the information about homotopy types.

We will assume throughout that the identity map $\id_{\disk}$ is a base point and therefore it will be omitted from notation.
For instance we denote the $n$-th homotopy group $\pi_n(\EAr{r},\id_{\disk})$ simply by $\pi_n\EAr{r}$ and so on.

\begin{theorem}\label{th:DA_EA_weak_he}
Let $\AFld$ be a \TC\ vector field on $\disk$.
Let $\DD$ denotes one of the groups $\DA$, $\DApl$, or $\DAd$, and $\EE$ be the corresponding semigroup $\EA$, $\EApl$, or $\EAd$. 
By $\DD^r$ (resp. $\EE^r$) we denote the topological space $\DD$ (resp. $\EE$) endowed with $\Wr{r}$ topology.
Then
\begin{enumerate}
 \item[\rm(1)] 
the inclusion $\DD^r\subset\EE^r$ is a \myemph{weak homotopy equivalence}%
\footnote{Recall that the map $i:\DD\to\EE$ is a  \myemph{weak homotopy equivalence} if for each $n\geq0$ the induced map $i_n:\pi_n(\DD,x)\to\pi_n(\EE,x)$ of homotopy sets (groups for $n\geq1$) is a bijection for each $x\in\DD$.} for $r\geq1$;
 \item[\rm(2)] 
in $\Wr{0}$ topology the induced map $\pi_0\DD^{0}\to\pi_0\EE^{0}$ is a \myemph{surjection};
 \item[\rm(3)]
for each $r\geq0$ the semigroup $\pi_0\EE^{r}$ is a \myemph{group} and any two path components of $\EE^r$ are homeomorphic each other.
\end{enumerate}
\end{theorem}

\begin{remark}\rm
In general a topological semigroup may have path components which are non homeomorphic each other.
For instance this often so for the semigroup of continuous maps $C(X,X)$ of a topological space $X$ with non-trivial homotopy groups, see e.g.~\cite{Hansen:JMOx:1974}.
\end{remark}

The next result describes the relative homotopy groups of the pair $(\EApl,\EAd)$.
\begin{theorem}\label{th:EAd_EApl}
Let $\AFld$ be a \TC\ vector field on $\disk$.
Then for each $r\geq0$
$$
\pi_n(\EAplr{r},\EAdr{r}) = 
\begin{cases}
\ZZZ, & n=1, \\
0,   & \text{otherwise}.
\end{cases}
$$
Hence the inclusion $\EAd\subset\EApl$ yields \myemph{isomorphisms} 
\begin{equation}\label{equ:pin_EAd_EApl}
\pi_n\EAdr{r} \to \pi_n\EAplr{r}, \qquad n\geq2,
\end{equation}
and we also have the following exact sequence:
\begin{multline}\label{equ:exseq_EAd_EApl}
0 \to \pi_1\EAdr{r} \to \pi_1\EAplr{r} \to 
\ZZZ \to \\ \to \pi_0\EAdr{r} \to \pi_0\EAplr{r} \to 0.
\end{multline}
\end{theorem}

Our next aim (see Theorem~\ref{th:EA_non-deg-cond} below) is to obtain some information about the homotopy groups of $\EA$, $\EApl$ and $\EAd$.
First we recall some definitions and preliminary results.

\subsection{Shift map}
Let $\AFlow:\disk\times \RRR \to \disk$ be the flow generated by $\AFld$ and 
$$\ShA:\Ci{\disk}{\RRR}\to \Ci{\disk}{\disk}$$
be the map defined by $$\ShA(\afunc)(z)=\AFlow(z,\afunc(z))$$
for $\afunc\in\Ci{\disk}{\RRR}$ and $z\in\disk$.
We will call $\ShA$ the \myemph{shift map} along orbits of $\AFld$ and denote its image in $\Ci{\disk}{\disk}$ by $\imShA$:
$$
\imShA := \ShA(\Ci{\disk}{\RRR}) \subset \Ci{\disk}{\disk}.
$$

\begin{lemma}
The following inclusions hold true:
\begin{equation}\label{equ:imSh_EidAr}
\imShA \;\subset\;  \EidAr{\infty} \;\subset\; \cdots \;\subset\; \EidAr{1} \;\subset\; \EidAr{0}.
\end{equation}
If $\imShA=\EidAr{r}$ for some $r=0,1,\ldots\infty$, then 
\begin{equation}\label{equ:Didinf_Dir}
\DidAr{\infty} = \cdots = \DidAr{r}
\end{equation}
whence the identity maps $\id:\DAr{\infty}\to\DAr{s}$ and $\id:\EAr{\infty}\to\EAr{s}$ for $s\geq r$ yield bijections:
$$
\pi_0\DAr{\infty} \ \approx \ \cdots \ \approx \ \pi_0\DAr{r}.
\qquad 
\pi_0\EAr{\infty} \ \approx \ \cdots \ \approx \ \pi_0\EAr{r},
$$
\end{lemma}
\begin{proof}
The first inclusion in~\eqref{equ:imSh_EidAr} follows from~\cite[Cor.~21]{Maks:Shifts} and the others are evident.
The fact that \eqref{equ:Didinf_Dir} is implied by the assumption $\imShA=\EidAr{r}$ is proved in~\cite{Maks:PartitPresDif}.
\end{proof}

The following Proposition~\ref{pr:info_on_TC} and Example~\ref{ex:homog_poly} describe some results about $\ker\jet$, $\imShA$, and $\EidAr{r}$ for \TC\ vector fields.
Then most complete information are given in the case when $\nabla\AFld$ is non-degenerate and when $\AFld$ is a ``reduced'' Hamiltonian vector field of some homogeneous polynomial on $\RRR^2$.

\begin{proposition}\label{pr:info_on_TC} 
Let $\AFld$ be a \TC\ vector field on $\disk$.

\NFZ\
If $\nabla\AFld=0$, then $\imShA\subset\ker\jet$.

\NFS\
Suppose that $\nabla\AFld$ is degenerate, but is not zero.
Then there are local coordinates at $\orig$ in which 
$\nabla\AFld = \left(\begin{smallmatrix}0 & a \\ 0 & 0 \end{smallmatrix} \right)$ for some $a\in\RRR\setminus\{0\}$.
Define the following subsets of $\GLRpl{2}$:
\begin{equation}\label{equ:Apl_Amn}
\begin{array}{lcl}
\Apl = \left\{ \left(\begin{smallmatrix}1 & d \\ 0 & 1 \end{smallmatrix} \right), \ d \in \RRR \right\},
& \qquad & 
\Amn = \left\{ \left(\begin{smallmatrix}-1 & d \\ 0 & -1 \end{smallmatrix} \right), \ d \in \RRR \right\}, \\ [2mm]
\Aopl = \left\{ \left(\begin{smallmatrix}1 & d \\ 0 & -1 \end{smallmatrix} \right), \ d \in \RRR \right\},
& \qquad & 
\Aomn = \left\{ \left(\begin{smallmatrix}-1 & d \\ 0 & +1 \end{smallmatrix} \right), \ d \in \RRR \right\},
\end{array}
\end{equation}
$$
\Ag =\Apl\cup\Amn \cup \Aopl \cup \Aomn.
$$
Then 
\begin{equation}\label{equ:nF_deg_Sh}
\jet(\imShA)=\Apl,
\quad
\jet(\EApl) \subset \Apl\cup\Amn,
\quad 
\jet(\EA) \subset \Ag.
\end{equation}

\NFN\
If $\nabla\AFld$ is non-degenerate, then there are local coordinates at $\orig$ in which $\AFld$ is given by
\begin{equation}\label{equ:Takens_nf}
\AFld(x,y) = \afunc(x,y) \left(-y \dd{x} + x \dd{y}\right) + \dAx \dd{x} + \dAy \dd{y},
\end{equation}
where $\afunc$ is a $\Cinf$ function such that $\afunc(\orig)\not=0$, and $\dAx$, $\dAy$ are \myemph{flat} at $\orig$.
Moreover,
$$\imShA = \EidAr{\infty}=\cdots = \EidAr{0} = \EApl = \jet^{-1}(\mathrm{SO}(2)),$$
$$\DidAr{\infty} = \cdots = \DidAr{0},$$
the inclusions $\DAd\subset\EAd$ and $\DApl\subset\EApl$ are homotopy equivalences with respect to $\Wr{\infty}$ topologies, $\DAd$ is contractible, and $\DApl$ is homotopy equivalent to the circle.

\medskip 

{\rm(4)}
Let $\theta:\disko\to(0,+\infty)$ be the function associating to each $z\in\disko$ its period $\theta(z)$ with respect to $\AFld$.
Then $\theta$ is $\Cinf$ on $\disko$ and we will call it the \myemph{period function}.

In the cases \NFZ\ and \NFS, i.e. when $\nabla\AFld$ is degenerate, 
$\lim\limits_{z\to\orig}\theta(z) = +\infty$ and thus $\theta$ can not be even continuously extended to all of $\disk$.
On the other hand in the case \NFN\ $\theta$ extends to a $\Cinf$ function on all of $\disk$.
\end{proposition}
\begin{proof}
Statement \NFZ\ is a particular case of \cite[Lm.~5.3]{Maks:jets}.
\NFS\ and (4) are established in~\cite{Maks:LocExtrPer}.

\NFN\ Representation~\eqref{equ:Takens_nf} is due to F.~Takens~\cite{Takens:AIF:1973}, and all other statements are proved in~\cite{Maks:LocExtrPer}.
Actually, F.~Takens shown that except for~\eqref{equ:Takens_nf} there is also an infinite series of normal forms for vector fields with ``rotation as $1$-jet'', but the orbits of these vector fields are non-closed, and so they are not \TC.
\end{proof}

\begin{example}\label{ex:homog_poly}
Let $\func:\RRR^2\to\RRR$ be a real homogeneous polynomial in two variables such that $\orig\in\RRR^2$ is a unique critical point of $\func$ being its global minimum.
Then we can write 
\begin{equation}\label{equ:homog_poly}
\func(x,y) = \prod_{j=1}^{k} Q_j^{\beta_j}(x,y),
\end{equation}
where every $Q_j$ is a positive definite quadratic form, $\beta_j \geq 1$, and
$$
\frac{Q_j}{Q_{j'}} \not= \mathrm{const} \ \text{for $j\not=j'$}.
$$
Then it is easy to see that $D= \prod\limits_{j=1}^{k} Q_j^{\beta_j-1}$ is the greatest common divisor of partial derivatives $\func'_{x}$ and $\func'_{y}$.
Let $\BFld= -\func'_{y}\tfrac{\partial}{\partial x} + \func'_{x}\tfrac{\partial}{\partial y}$ be the Hamiltonian vector field of $\func$ and 
$$\AFld= -(\func'_{y}/D)\,\tfrac{\partial}{\partial x} \,+\, (\func'_{x}/D)\,\tfrac{\partial}{\partial y}.$$
Then the coordinate function of $\AFld$ are relatively prime in the ring $\RRR[x,y]$.
We will call $\AFld$ the \myemph{reduced} Hamiltonian vector field for $\func$.

Fix $\eps>0$ and put $\Vman=\func^{-1}[0,\eps]$.
Then $\AFld$ is a \TC\ vector field on $\Vman$ with singularity at $\orig$.

If $k=1$, then $\nabla\AFld$ is non-degenerate and a description of $\pi_0\EAplr{\infty}$ and $\EidAr{\infty}$ is given by \NFN\ of Proposition~\ref{pr:info_on_TC}.

If $k\geq2$, then $\nabla\AFld=0$.
In this case, see \cite{Maks:PartitPresDif,Maks:ImSh},
$$\ker\jet  = \imShA = \EidAr{\infty} = \cdots = \EidAr{1} \not=\EidAr{0}=\EApl,$$ 
$\EidAr{\infty}$ is contractible with respect to $\Wr{\infty}$ topology, and $\pi_0\EAplr{\infty}\approx \ZZZ_{2n}$ for some $n\geq1$.
\end{example}

\medskip

Now we can formulate our last result Theorem~\ref{th:EA_non-deg-cond}.
It gives some information about \myemph{weak} homotopy types of $\EA$ and $\EAd$ under certain assumptions on $\AFld$.
The main assumption is the following one:
\begin{equation}\label{equ:non-deg-condition}
\ker\jet\ \subset\ \imShA.
\end{equation}
It means that for every $\hdif\in\EA$ with \myemph{the identity map as $1$-jet} there exists a $\Cinf$ shift function on all of $\disk$.

\begin{theorem}\label{th:EA_non-deg-cond}
Let $\AFld$ be a \TC\ vector field on $\disk$ such that $\nabla\AFld$ is degenerate and $\ker\jet\subset\imShA$.
Let also $r\geq1$. 
Then the following statements hold true.

\NFZ\
If $\nabla\AFld=0$, then $\imShA = \ker\jet$.
Let $\id\in\GLR{2}$ be the unit matrix.
If in addition the path component of $\id$ in the image $\LA=\jet(\EA)$ of $\jet$ coincides with $\{\id\}$ (e.g. when $\LA$ is discrete), then $\imShA=\EidAr{1}$.

Hence $\jet$ induces isomorphisms 
\begin{equation}\label{equ:pi0_EApl_EA}
\pi_0\EAplr{r} \approx \LA\cap\GLRpl{2},
\qquad 
\pi_0\EAr{r} \approx \LA.
\end{equation}

\medskip 

\NFS\
If $\nabla\AFld =\left(\begin{smallmatrix} 0 & a \\ 0 & 0  \end{smallmatrix}  \right)$ for some $a\not=0$, then 
\begin{equation}\label{equ:imShA_jApl}
\imShA = \EidAr{\infty} = \cdots = \EidAr{1}  = \jet^{-1}(\Apl),
\end{equation}
whence $\jet$ yields a monomorphism, see~\eqref{equ:Apl_Amn},
$$
\pi_0 \EAr{\infty} \ \longrightarrow \  \pi_0\Ag \approx \ZZZ_2 \oplus \ZZZ_2.
$$

\medskip 

{\rm(3)}
The inclusion $\EidAdr{r}\subset\EidAplr{r}$ between the identity path components is a weak homotopy equivalence, whence  from Theorem~\ref{th:EAd_EApl} we have the isomorphisms 
$$\pi_n\EAdr{r} \approx \pi_n\EAplr{r}, \qquad n\geq1,$$
and following exact sequence:
\begin{equation}\label{equ:pi0_EAd_EApl}
 0 \to \ZZZ \to \pi_0\EAdr{r} \to  \pi_0\EAplr{r} \to 0.
\end{equation}

\medskip 

{\rm(4)}
Suppose that the image $\LA$ of $\jet$ is finite.
Then $\pi_0\EAdr{r}\approx\ZZZ$, $\pi_0\EAplr{r}\approx \ZZZ_n$ for some $n\geq0$,
and~\eqref{equ:pi0_EAd_EApl} has the following form
$$
0 \to \ZZZ \xrightarrow{~~\cdot n~~} \ZZZ \xrightarrow{~~\mathrm{mod}\ n~~} \ZZZ_{n}  \to 0.
$$
If $\EA\not=\EApl$, then $\pi_0\EAr{r}\approx \mathbb{D}_n$ is a dihedral group.
\end{theorem}

The proof of Theorems~\ref{th:DA_EA_weak_he}, \ref{th:EAd_EApl}, and \ref{th:EA_non-deg-cond} will be given in \S\S\ref{sect:proof:th:DA_EA_weak_he}-\ref{sect:proof:th:EA_non-deg-cond}.
All of them are based on results of~\cite{Maks:ImSh} described in~\S\ref{sect:sh_func_Kr_def} about existence and uniqueness of shift functions for deformations in $\EApl$.

\section{The inclusion $\DA\subset\EA$}\label{sect:DA_EA}
The aim of this section is to prove Lemma~\ref{lm:change_to_diff} which allows to change elements of $\EA$ outside some neighbourhood of $\orig$ to produce diffeomorphisms.

Let $\AFld$ be a \TC\ vector field of $\disk$ and $\func:\disk\to[0,1]$ be a first strong intergal for $\AFld$, see~Definition~\ref{def:first_strong_int}.
For every $c\in(0,1]$ put $\Uman_{c}=\func^{-1}[0,c]$.
Then $\Uman_{c}$ is invariant with respect to $\AFld$.

\begin{lemma}\label{lm:change_to_diff}
Let $\hdif\in\EA$.
Then there exists $\gdif\in\DA$ such that $\hdif=\gdif$ on some neighbourhood of $\orig$.
\end{lemma}
\begin{proof}
By definition $\hdif\in\EA$ is a diffeomorphism at $\orig$, whence there exists $\eps\in(0,1/2)$ such that $\hdif:\Uman_{2\eps} \to \Uman_{2\eps}$ is a diffeomorphism.
Fix any $C^{\infty}$-diffeomorphism $\mu:[0,2\eps]\to[0,1]$ such that $\mu=\id$ on $[0,\eps]$, see Figure~\ref{fig:func_mu}.
\begin{figure}[ht]
\includegraphics[height=1.5cm]{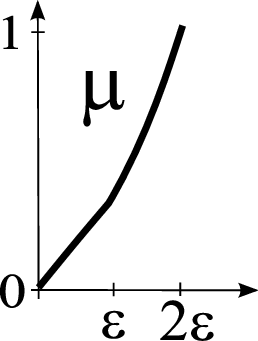}
\caption{}
\protect\label{fig:func_mu}
\end{figure}

We will now construct a diffeomorphism $\psi:\Uman_{2\eps}\to\disk$ fixed on $\Uman_{\eps}$ and such that $\func \circ \psi = \mu \circ \func$, i.e. it makes commutative the following diagram:
$$
\begin{CD}
\Uman_{2\eps} @>{\psi}>> \disk \\
@V{\func}VV @VV{\func}V \\
[0,2\eps] @>{\mu}>> [0,1]
\end{CD}
$$
It follows that if $c\in[0,2\eps]$ and $\orb=\func^{-1}(c)$ is an orbit of $\AFld$, then $\psi(\orb) = \func^{-1}(\mu(c))$ is also an orbit of $\AFld$.
Then we can define a diffeomorphism $\gdif:\disk\to\disk$ by 
$$
\begin{CD}
\gdif = \psi\circ\hdif|_{\Uman_{2\eps}}\circ\psi^{-1}:
\disk @>{\psi^{-1}}>> \Uman_{2\eps} @>{\hdif}>>  \Uman_{2\eps} @>{\psi}>> \disk.
\end{CD}
$$
Then $\gdif\in\DA$ and since $\psi$ is fixed on $\Uman_{\eps}$, it follows that $\gdif=\hdif$ on $\Uman_{\eps}$.

Construction of $\psi$ is similar to~\cite[Lm.~5.1.3]{Maks:BSM:2006}.
Consider the gradient vector field $\grad\func$ of $\func$ defined on $\disko$, and let $(\Phi_t)$ be the local flow of $\grad\func$.
Let $z\in\Uman_{2\eps}$ and $\gamma$ be the orbit of $z$ with respect to $\Phi$.
Then $\gamma$ intersect the level-set $\func^{-1}(\mu(\func(z)))$ at a unique point $\psi(z)$, see Figure~\ref{fig:psi_constr}.
\begin{figure}[ht]
\includegraphics[height=3cm]{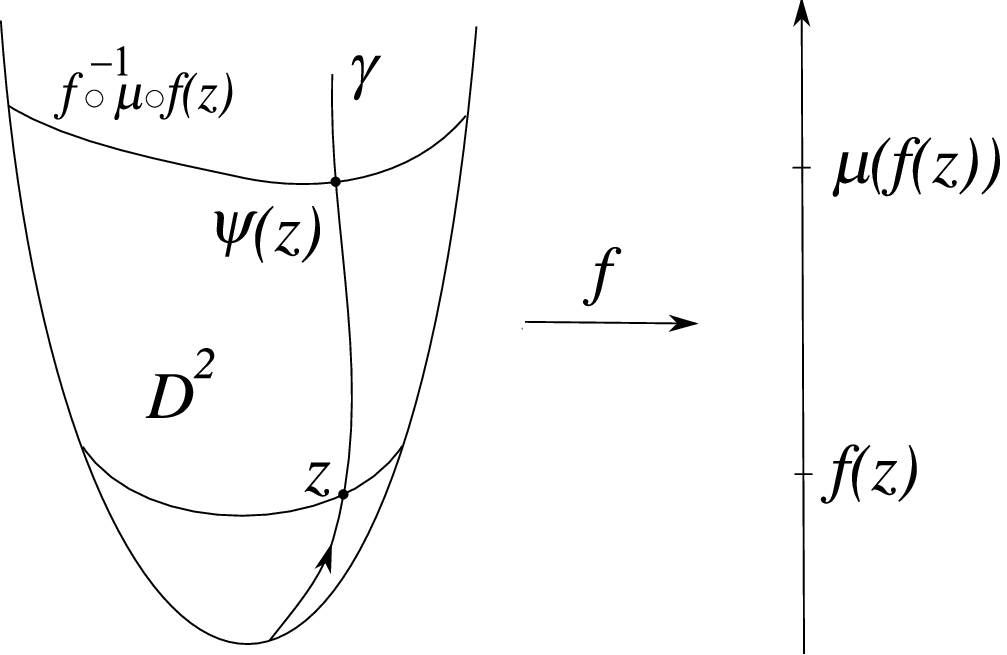}
\caption{}
\protect\label{fig:psi_constr}
\end{figure}
Similarly to~\cite{Maks:BSM:2006} it can be shown that the correspondence $z\mapsto\psi(z)$ is a diffeomorphism of $\Uman_{2\eps}\to\disk$ if and only if so is $\mu$.
\end{proof}

\section{Shift functions}\label{sect:shift-func}
Let $\Mman$ be a smooth ($\Cinf$) manifold, $\AFld$ be a $\Cinf$ vector field on $\Mman$ generating a flow $\AFlow:\Mman\times\RRR\to\Mman$, and 
$\ShA:\Ci{\Mman}{\RRR}\to\Ci{\Mman}{\Mman}$ be the \myemph{shift map along orbits of $\AFld$} defined by $\ShA(\afunc)(z)=\AFlow(z,\afunc(z))$.

If a subset $\Vman\subset\Mman$, a function $\afunc:\Vman\to\RRR$, and a map $\hdif:\Mman\to\Mman$ are such that $\hdif(z)=\AFlow(z,\afunc(z))$, then we will say that $\afunc$ is \myemph{shift function} for $\hdif$ on $\Vman$, and that the restriction $\hdif|_{\Vman}$ is in turn a \myemph{shift along orbits of $\AFld$ via $\afunc$}.

For a $\Cinf$ function $\afunc:\Mman\to\RRR$ we will denote by $\AFld(\afunc)$ the Lie derivative of $\afunc$ along $\AFld$.

\begin{lemma}\label{lm:when_h_in_DApl}{\rm\cite[Th.~19]{Maks:Shifts}}
Let $\Vman\subset\Mman$ be an open subset, $\afunc:\Vman\to\RRR$ a $\Cinf$ function, and $\hdif:\Vman\to\Mman$ be a map defined by $\hdif(z)=\AFlow(z,\afunc(z))$.
Then $\hdif$ is a local diffeomorphism at some $z_0\in\Mman$ if and only if $\AFld(\afunc)(z_0)\not=-1$.
\end{lemma}

\begin{lemma}\label{lm:shift-func-for-frac}{\rm\cite{Maks:Shifts}}
Let $\afunc_{\gdif},\afunc_{\hdif}, \afunc_{\kdif}:\Mman\to\RRR$ be $\Cinf$ functions and 
$$
\gdif=\ShA(\afunc_{\gdif}),
\qquad 
\hdif=\ShA(\afunc_{\hdif}),
\qquad 
\kdif=\ShA(\afunc_{\kdif})
$$
be the corresponding shifts.
Suppose also that $\kdif$ is a diffeomorphism.
Then the following functions 
$$
\afunc_{\gdif\circ\hdif} = \afunc_{\gdif}\circ\hdif + \afunc_{\hdif},
\qquad
\afunc_{\kdif^{-1}} =-\afunc_{\kdif}\circ\kdif^{-1},
$$
$$\afunc_{\gdif\circ\kdif^{-1}}=(\afunc_{\gdif}-\afunc_{\kdif})\circ\kdif^{-1}$$ 
are $\Cinf$ shift functions for $\gdif\circ\hdif$, $\kdif^{-1}$, and $\gdif\circ\kdif^{-1}$ respectively.
\end{lemma}
\begin{proof}
Formulae for $\afunc_{\gdif\circ\hdif}$ and $\afunc_{\kdif^{-1}}$ coincide with~\cite[Eq.~(8),(9)]{Maks:Shifts}.
They also imply formula for $\afunc_{\gdif\circ\kdif^{-1}}$.
\end{proof}

\section{Shift functions for $\EApl$}\label{sect:shift_func_EApl}
Let $\tdisk=\{(\phi,\rrho)\in\RRR^2 \ : \ 0\leq \rrho \leq 1 \}$ be a closed strip,
$$\tdisko=\{(\phi,\rrho)\in\RRR^2 \ : \ 0< \rrho \leq 1 \}=\tdisk\setminus\{\rrho=0\}$$ be a half-closed strip in $\RRR^2$, and $P:\tdisk\to\disk$ be the map given by $$P(\rrho,\phi)=(\rrho\cos\phi,\rrho\sin\phi).$$

Then $P(\tdisko)=\disko$ and the restriction $P:\tdisko\to\disko$ is $\ZZZ$-covering map such that the corresponding group of covering transformations is generated by the following map 
$$\eta:\tdisk\to\tdisk, \qquad \eta(\phi,\rrho)=(\phi+2\pi,\rrho).$$

It follows that every $\Cinf$ map $\hdif:\disko\to\disko$ lifts to a $P$-equivariant (i.e. commuting with $\eta$) map $\thdif:\tdisko\to\tdisko$ such that 
$$P\circ\thdif = \hdif\circ P.$$
Such $\thdif$ is not unique and can be replaced with $\thdif\circ\eta^n=\eta^{n}\circ\thdif$ for any $n\in\ZZZ$.

\begin{remark}
It is well-known that if $\hdif:\disk\to\disk$ is a $\Cinf$ map being a local diffeomorphism at $\orig$, and such that $\hdif^{-1}(\orig)=\orig$, then $\thdif$ extends to a $\Cinf$ map $\thdif:\tdisk\to\tdisk$ being a diffeomorphism near $\phi$-axis $\{\rrho=0\}$.
We will not use this fact in the present paper.
\end{remark}

Let $\AFld$ be a \TC\ vector field on $\disk$.
Since $\AFld$ is non-singular on $\disko$, $\AFld$ lifts to a unique vector field $\BFld$ on $\tdisko$ such that $\AFld\circ P = TP \circ \BFld$.

It is easy to see that every orbit $\torb$ of $\BFld$ is non-closed, its image $\orb=P(\torb)$ is an orbit of $\AFld$, and the map $P:\torb\to\orb$ is a $\ZZZ$-covering map.

Let $\BFlow:\tdisko\times\RRR\to\tdisko$ be the flow generated by $\BFld$, then we have the following commutative diagram: 
\begin{equation}\label{equ:lift_flow}
\begin{CD}
\tdisko\times\RRR @>{\BFlow}>> \tdisko \\
@V{P\times \id_{\RRR}}VV @VV{P}V \\
(\disko)\times\RRR @>{\AFlow}>> \disko
\end{CD}
\end{equation}
In other words, $\AFlow_t \circ P(\tz) = P \circ \BFlow_t(\tz)$ for all $\tz\in\tdisko$ and $t\in\RRR$.

In particular, if $\afunc:\disk\to\RRR$ is a $\Cinf$ function and $\hdif=\ShA(\afunc)$, i.e. $\hdif(z)=\AFlow(z,\afunc(z))$, then the map $\thdif:\tdisk\to\tdisk$ given by $\BFlow(\tz,\afunc\circ P(\tz))$ is a lifting of $\hdif$.
Indeed,
\begin{equation}\label{equ:lift_shift_func}
\hdif\circ P(\tz) = 
\AFlow(P(\tz),\afunc\circ P(\tz)) = 
P \circ \BFlow(\tz,\afunc\circ P(\tz)) = P\circ\thdif(\tz).
\end{equation}

\begin{lemma}\label{lm:exist_shift_func}
Let $\hdif\in\EApl$.
Then there exists a $\Cinf$ shift function $\bfunc:\disko\to\RRR$ for $\hdif$ on $\disko$, i.e. $\hdif(z)=\AFlow(z,\bfunc(z))$ for $z\in\disko$.
Any other $\Cinf$ shift function for $\hdif$ on $\disko$ is given by $\bfunc + n\theta$ for some $n\in\ZZZ$.

If $\nabla\AFld$ is degenerate, then any $\hdif\in\EApl$ has at most one $\Cinf$ shift function defined on all of $\disk$.
\end{lemma}
\begin{proof}
By definition $\hdif^{-1}(\orig)=\orig$ and $\hdif$ is a local diffeomorphism at $\orig$.
Then, as noted above, there exists a $\Cinf$ lifting $\thdif:\tdisko\to\tdisko$ of $\hdif$ such that $P\circ \thdif=\hdif\circ P$.
Moreover, $\hdif$ preserves orbits of $\AFld$, whence $\hdif(\torb)=\torb$ for each orbit $\torb$ of $\BFld$.

Since the orbits of $\BFld$ are non-closed there exists a unique $\Cinf$ shift function $\tbfunc:\tdisko\to\RRR$ for $\thdif$, i.e. $\thdif(\tz)=\BFlow(\tz,\tafunc(\tz))$ for all $\tz\in\tdisko$.
Also notice that $\BFlow_t$ and $\thdif$ are $\ZZZ$-equivariant.
This easily implies that $\tbfunc$ is $\ZZZ$-invariant, whence it defines a unique $\Cinf$ function $\bfunc:\disko\to\RRR$ such that $\tbfunc= \bfunc\circ P$.
Then it follows from~\eqref{equ:lift_shift_func} that $\bfunc$ is a shift function for $\hdif$ with respect to $\AFld$.

\medskip

Suppose that $\afunc:\disko\to\RRR$ is another $\Cinf$ shift function for $\hdif$ on $\disko$.
Then $\hdif(z)=\AFlow(z,\afunc(z))=\AFlow(z,\bfunc(z))$ for $z\not=0$, whence the difference $\afunc(z)-\bfunc(z)$ is a certain integer multiple of the period $\theta(z)$ of $z$.
Since $\theta$ and $\afunc-\bfunc$ are $\Cinf$ on $\disko$, it follows that $\afunc-\bfunc = n\theta$ for some $n\in\ZZZ$.

\medskip

Finally, suppose that $\nabla\AFld$ is degenerate, and $\afunc,\bfunc:\disk\to\RRR$ are two $\Cinf$ shift function for $\hdif$ defined on all of $\disk$.
Then they also shift functions for $\hdif$ on $\disko$, whence $\afunc-\bfunc=n\theta$ for some $n\in\ZZZ$.
But by Proposition~\ref{pr:info_on_TC} $\lim\limits_{z\to\orig}\theta(z)=+\infty$, while $\afunc-\bfunc$ is $\Cinf$ on all of $\disk$.
Hence $n=0$, i.e. $\afunc=\bfunc$.
\end{proof}

\section{$(\Ksp,r)$-deformations}\label{sect:Kr_deform}
Let $A$, $B$ be smooth manifolds.
Then the space $C^{\infty}(A,B)$ admits a series $\{\Wr{r}\}_{r=0}^{\infty}$ of \myemph{weak} topologies, see~\cite{Hirsch:DiffTop}.
Topology $\Wr{0}$ coincides with the \myemph{compact open} one.
Let $J^{r}(A,B)$, $(r<\infty)$, be the manifold of $r$-jets of maps $A\to B$.
Then there is a natural inclusion $i_{r}:C^{\infty}(A,B) \subset C^{\infty}(A,J^{r}(A,B))$ associating to each $\func:A\to B$ its $r$-jet prolongation $j^{r}(f):A\to J^{r}(A,B)$.
Endow $C^{\infty}(A,J^{r}(A,B))$ with $\Wr{0}$ topology.
Then the topology on $C^{\infty}(A,B)$ induced by $i_{r}$ is called $\Wr{r}$ topology.
Finally, the topology $\Wr{\infty}$ is generated by all $\Wr{r}$ for $0\leq r <\infty$.

Let $\XX\subset C^{\infty}(A,B)$ be a subset, $\Ksp$ be a Hausdorff, locally compact topological space, and $\omega:\Ksp\to \XX$ be a map.
Then $\omega$ induces the following mapping $\Omega:\Ksp\times A \to B$ defined by $\Omega(a,k) = \omega(k)(a)$.
Conversely, every map $\Omega:\Ksp\times A \to B$ such that $\Omega(k,\cdot):A\to B$ belongs to $\XX$ induces a map $\omega:\Ksp\to\XX$.

Endow $\XX$ with the induced $\Wr{0}$ topology.
Then it is well known, e.g.~\cite[\S44.IV]{Kuratowski:2}, that $\omega$ is continuous if and only if $\Omega$ is so.

\begin{definition}
Let $r=0,\ldots,\infty$.
Then a map $\Omega:\Ksp\times A\to B$ will be called a \myemph{$(K,r)$-deformation in $\XX$} if $\Omega_{k}\in \XX$ for all $k\in\Ksp$ and the induced map $\omega:K\to\XX$ is continuous whenever $\XX$ is endowed with $\Wr{r}$ topology.
In other words the map $j^r:\Ksp\times A \to J^r(A,B)$ associating to each $(k,a)\in\Ksp\times A$ the $r$-jet prolongation $j^{r}\Omega_k(a)$ of $\Omega_k$ at $a$ is continuous.

If $\Ksp=[0,1]$ then a $(\Ksp,r)$-deformation will be called an \myemph{$r$-homotopy}.
\end{definition}

\section{Shift functions for $(\Ksp,r)$-deformations}\label{sect:sh_func_Kr_def}
Let $\Ksp$ be a Hausdorff, locally compact, and path connected topological space, $\AFld$ be a \TC\ vector field on $\disk$.
Let also  
$$\omega:\Ksp\to\EAplr{r}$$ be a continuous map into some $\Wr{r}$ topology of $\EApl$, and 
\begin{equation}\label{equ:Kr_deform_for_omega}
\Omega:\Ksp\times\disk\to\disk,
\qquad
\Omega(k,z) = \omega(k)(z).
\end{equation}
be the corresponding $(\Ksp,r)$-deformation in $\EApl$, so $\Omega_k\in\EApl$ for all $k\in\Ksp$.

Then by Lemma~\ref{lm:exist_shift_func} for each $k\in\Ksp$ the map $\Omega_k$ has a (not unique) $\Cinf$ shift function $\Lambda_k$ defined on $\disko$.
Thus we can define a map $\Lambda:\Ksp\times(\disko)\to\RRR$ by $\Lambda(k,z)=\Lambda_k(z)$ which in general is not even continuous, though it is $\Cinf$ for each $k$.

\begin{definition}\em
A $(\Ksp,r)$-deformation $\Lambda:\Ksp\times(\disko)\to\RRR$ satisfying
\begin{equation}\label{equ:Omega_ShA_Lambda}
\Omega(k,z) = \AFlow(z,\Lambda(k,z))
\qquad \forall \,(k,z)\in \Ksp \times (\disko).
\end{equation}
will be called a \myemph{shift function} for the $(\Ksp,r)$-deformation $\Omega$.
\end{definition}

The following lemma is a particular case of results of~\cite{Maks:ImSh}, see also~\cite[Th.~25]{Maks:Shifts}.
It describes existence and uniqueness of shift functions for $(\Ksp,r)$-deformations.

\begin{lemma}\label{lm:ext_shift_func_via_homotopy}{\rm\cite{Maks:ImSh}}
Let $k_0\in\Ksp$ and $\Lambda_{k_0}$ be any $\Cinf$ shift function for $\Omega_{k_0}$.
Then there exists at most one shift function $\Lambda:\Ksp\times(\disko)\to\RRR$ for $\Omega$ such that $\Lambda(k_0,z)=\Lambda_{k_0}(z)$.

Moreover, if $\Ksp$ is simply connected, i.e.\! $\pi_1\Ksp=0$, then any shift function $\Lambda_{k_0}$ for $\Omega_{k_0}$ uniquely extends to a shift function $$\Lambda:\Ksp\times(\disko)\to\RRR$$ for $\Omega$.
\qed
\end{lemma}
This lemma will be used in the proof of Theorems~\ref{th:DA_EA_weak_he} and~\ref{th:EAd_EApl}.
For the proof of Theorem~\ref{th:EA_non-deg-cond} we will also need the following Lemmas~\ref{lm:ext_shift_func} and~\ref{lm:deform_in_imShA}.

Suppose now that $\omega(\Ksp)\subset\imShA$, that is for each $k\in\Ksp$ the map $\Omega_k$ has a $\Cinf$ shift function defined on all of $\disk$.
Let $\Lambda:\Ksp\times(\disko)\to\RRR$ be a shift function for $\Omega$ such that $\Lambda_{k_0}$ for some $k_0\in\Ksp$ smoothly extends to all of $\disk$.
The following lemma gives sufficient conditions when every other shift function $\Lambda_k=\Lambda(k,\cdot)$ smoothly extends to all of $\disk$.
Again it is a particular case of results of~\cite{Maks:ImSh}.
\begin{lemma}\label{lm:ext_shift_func}{\rm\cite{Maks:ImSh}}
Let $\Omega:\Ksp\times\disk\to\disk$ be a $(\Ksp,r)$-deformation admitting a 
shift function $\Lambda:\Ksp \times (\disko) \to \RRR$.
Suppose that 
\begin{enumerate}
 \item[\rm(i)] $\ker\jet\subset\imShA$,
 \item[\rm(ii)] $\Omega_0=\id_{\disk}$ for some $k_0\in\Ksp$, and
 \item[\rm(iii)] $\jt{\Omega_k}=\id$, i.e. $\Omega_k\in\ker\jet \subset\imShA$, for all $k\in\Ksp$.
\end{enumerate}
Then for each $k\in\Ksp$ the function $\Lambda_k:\disko\to\RRR$ extends to a $\Cinf$ function on all of $\disk$, though the induced function $\Lambda:\Ksp\times\disk\to\RRR$ is not necessarily continuous.
\qed
\end{lemma}

Finally, we present a sufficient condition when a map into $\imShA$ can be deformed into $\ker\jet$.
\begin{lemma}\label{lm:deform_in_imShA}
Let $\Ksp$ be path connected and simply connected, $r\geq1$, and $\omega:\Ksp\to\imShA$ be a continuous map into $\Wr{r}$ topology of $\imShA$.
Suppose $\nabla\AFld$ is degenerate and $\ker\jet\subset\imShA$.
Then there exists a homotopy $\Bhom:I\times \Ksp\to\imShA$ such that $\Bhom_0=\omega$, $\Bhom_1(\Ksp)\subset\ker\jet$, and $\Bhom_t(k)=\omega(k)$ for all $k$ such that $\omega(k)\in\ker\jet$.
\end{lemma}
\begin{proof}
If $\nabla\AFld=0$, then $\imShA=\ker\jet$ and there is nothing to prove.

Suppose that $\nabla\AFld=\left(\begin{smallmatrix} 0 & a \\ 0 & 0  \end{smallmatrix}  \right)$ for some $a\not=0$.
Let $\Omega:\Ksp\times\disk\to\disk$ be the corresponding $(\Ksp,r)$-deformation in $\imShA$.
Then 
$$
\jt{\Omega_k}=\left(\begin{smallmatrix} 1 & a\, \tau(k) \\ 0 & 1  \end{smallmatrix} \right), 
\qquad 
k\in I,
$$
for some $\tau(k)\in \RRR$.
Since $\Omega$ is an $r$-homotopy with $r\geq1$, it follows that the function $\tau:\Ksp\to\RRR$ is continuous.
Moreover, $\tau(k)=0$ if and only if $\Omega_k\in\ker\jet$.

Define the homotopy $\Bhom:I\times\Ksp\to\imShA$ by 
$$
\Bhom(t,k)(z) = \AFlow(\Omega(k,z), -t\tau(k)).
$$
Then it is easy to see that $\Bhom$ satisfies the statement of our lemma.
\end{proof}

\section{Deformations in $\EApl$}\label{sect:deform_in_EApl}
In this section we prove the key Proposition~\ref{pr:deform_in_EApl} which will imply Theorems~\ref{th:DA_EA_weak_he}, \ref{th:EAd_EApl}, and~\ref{th:EA_non-deg-cond}.

Let $\Ksp$ be a Hausdorff, locally compact topological space and 
$$\omega:\Ksp\to\EAplr{r}$$ be a continuous map into some $\Wr{r}$ topology of $\EApl$.
Our aim is to show that under certain mild assumptions $\omega$ is homotopic to a map into $\DAd=\DApl\cap\EAd$ so that the intersections of $\omega(\Ksp)$ with $\DAd$, $\DApl$ and $\EAd$ remain in the corresponding spaces during all the homotopy.
More precisely the following result holds true:
\begin{proposition}\label{pr:deform_in_EApl}
Suppose that either 
\begin{enumerate}
 \item[\rm(i)]
$\Ksp$ is a point and $r\geq0$, or
 \item[\rm(ii)] 
$\Ksp$ is compact, path connected, and simply connected, and $r\geq1$.
\end{enumerate}

Let $\Ldf\subset\Ksp$ be a (possibly empty) subset such that $\omega(\Ldf)\subset\DApl$, and $\Kpc\subset\Ksp$ be a connected subset such that $\omega(\Kpc)\subset\EAd$.
Thus we can regard $\omega$ as a map of triples
$$
\omega:(\Ksp; \Ldf, \Kpc') \to (\EAplr{r};\DAplr{r},\EAdr{r}).
$$
Then there exists a homotopy of triples 
$$\Ahom_t:(\Ksp; \Ldf, \Kpc')\to(\EAplr{r};\DAplr{r},\EAdr{r}), \quad t\in I,$$
such that 
\begin{equation}\label{equ:Ahom_01}
\Ahom_0=\omega
\qquad  \text{and} \qquad 
\Ahom_1(\Ksp)\subset\DAd.
\end{equation}
The phrase \myemph{homotopy of triples} means that 
\begin{equation}\label{equ:omegat_L_Dapl__K_EAd}
\Ahom_t(\Ldf)\subset \DApl,
\qquad
\Ahom_t(\Kpc)\subset \EAd,
\end{equation}
and therefore $\Ahom_t(\Ldf \cap \Kpc)\subset \DApl \cap \EAd = \DAd$ for all $t\in I$.
\end{proposition}

The proof will be given at the end of this section.
Let 
$$ 
\Omega:\Ksp\times\disk\to\disk,
\qquad
\Omega(k,z) = \omega(k)(z)
$$ 
be the corresponding $(\Ksp,r)$-deformation in $\EApl$.
Then by Lemma~\ref{lm:ext_shift_func_via_homotopy} there exists a shift function $\Lambda:\Ksp\times(\disko)\to\RRR$ for $\Omega$.
The deformation of $\Omega_k$ we will be produced via a deformation of $\Lambda_k$.

Let $a,b\in(0,1)$ be such that $a<b$, $\func:\disk\to[0,1]$ be the first strong integral for $\AFld$, see Definition~\ref{def:first_strong_int}, and $\nu:[0,1]\to[0,1]$ be a $\Cinf$ function such that $\nu\bigl[0,a]=1$ and $\nu[b,1]=0$.
Define the following function $\afunc:\Ksp\times(\disko)\to\RRR$ by
\begin{equation}\label{equ:a_nuf_L}
\afunc(k,z) = \nu\circ\func(z) \; \cdot \; \Lambda(k,z), \qquad (k,z)\in\Ksp\times(\disko),
\end{equation}
the map $\Omega':\Ksp\times\disk\to\disk$ by
\begin{equation}\label{equ:g_shift_afunc}
\Omega'(k,z)=
\begin{cases}
\AFlow(z,\afunc(k,z)), & z \not=\orig, \\
\orig, & z=\orig,
\end{cases}
\end{equation}
and a homotopy $A:I \times\Ksp\times \disk \to \disk$ by
\begin{equation}\label{equ:homotopy_to_diffeo}
A(t,k,z) = 
\begin{cases}
\AFlow(z, \, (1-t)\afunc(k,z)+ t \Lambda(k,z) ), & z\not=\orig, \\
\orig, & z=\orig.
\end{cases}
\end{equation}

\begin{lemma}\label{lm:prop_homot_A}
For $t,k\in I\times \Ksp$ denote 
$A_t=A(t,\cdot,\cdot):\Ksp\times\disk\to\disk$ and
$A_{t,k}=A(t,k,\cdot):\disk\to\disk$.
Then 
\begin{enumerate}
 \item[\rm(a)] 
$A_0=\Omega$, $A_{1}=\Omega'$, and
$A_t$ is a $(\Ksp,r)$-deformation in $\EA$ for each $t\in I$.
 \item[\rm(b)]
$\Omega'_{k}$ is fixed on $\disk\setminus\Uman_{b}$ for all $k\in\Ksp$.
In particular, $A_1=\Omega'$ is a deformation in $\EAd$.
 \item[\rm(c)]
If for some $(k,z)\in \Ksp\times \disk$ the map $\Omega_{k}$ is a local diffeomorphism at $z$, then so is $A_{t,k}$ for each $t\in I$.
 \item[\rm(d)]
Denote $\Zsp=\Lambda^{-1}(0)  \subset \Ksp\times (\disko)$.
Thus $\Omega(k,z)=z$ for all $(k,z)\in\Zsp$.
Then $$A(t,k,z)=z, \qquad \forall\,t\in I, \ (k,z)\in\Zsp.$$
 \item[\rm(e)]
Let $\Kpc\subset\Ksp$ be a connected subset such that $\Omega_{k}$ is fixed on $\partial\disk$ for each $k\in\Kpc$ and $\Lambda_{k_0}|_{\partial\disk}=0$ 
for some $k_0\in\Kpc$.
Then $A_{t,k}$ is also fixed on $\partial\disk$ for all $(t,k)\in I\times\Kpc$.
\end{enumerate}

Thus $A$ induces a homotopy
\begin{equation}\label{equ:omega_t_induced_by_A}
\Ahom_t:\Ksp\to\EAplr{r}, \qquad \Ahom_t(k)(z) = A(t,k,z)
\end{equation}
such that $\Ahom_0=\omega$, and $\Ahom_1(\Ksp)\subset\EAd$.
\end{lemma}
\begin{proof}
Statements (a) and (b) follows from~\eqref{equ:a_nuf_L}-\eqref{equ:homotopy_to_diffeo}.

(c) 
Denote 
\begin{equation}\label{equ:beta_func}
\bfunc_{t,k}(z) = (1-t)\afunc(k,z)+ t \Lambda(k,z) =
\bigl( (1-t)\nu\circ\func(z) + t \bigr) \cdot \Lambda_{k}(z).
\end{equation}
Then by~\eqref{equ:homotopy_to_diffeo} $\bfunc_{t,k}$ is a shift function for $A_{t,k}$ on $\disko$.

The assumption that $\Omega_{k}$ is a local diffeomorphism at $z$ means that 
\begin{equation}\label{equ:F_L_ge_1}
\AFld(\Lambda_{k})(z)>-1,
\end{equation}
see Lemma~\ref{lm:when_h_in_DApl}.
Therefore by that lemma it suffices to verify that $\AFld(\bfunc_{t,k})(z)>-1$ for all $t\in I$.

Notice that 
$$ 
\AFld\bigl( ((1-t)\nu\circ\func + t ) \cdot \Lambda_{k}\bigr)=
 (1-t)\AFld(\nu\circ\func) +  \bigl((1-t)\nu\circ\func + t \bigr) \AFld(\Lambda_{k}).
$$ 
The first summand is zero since $\func$ and therefore $\nu\circ\func$ are constant along orbits of $\AFld$.
Moreover, $0\leq \nu(z) \leq 1$, whence we get from~\eqref{equ:F_L_ge_1} that the second summand is $>-1$.
Hence $\AFld(\bfunc_{t,k})(z) >-1$ for all $k\in\Ksp$.

(d)
If $\Lambda(k,z)=0$ for some $(k,z)\in \Ksp\times(\disko)$, then by~\eqref{equ:beta_func} $\bfunc_{t,k}(z)=0$, whence 
$A(t,k,z) = \AFlow(z,\bfunc_{t,k}(z)) =\AFlow(z,0)=z$.

(e)
If $\Omega_k$ is fixed on $\partial\disk$ for some $k\in\Ksp$, then $\Lambda_{k}$ takes on $\partial\disk$ a constant value:
$$
\Lambda_k(\partial\disk) \ = \ n_k \cdot \theta(\partial\disk)
$$
for some $n_k\in\ZZZ$.
Since $\Kpc$ is connected, $\Lambda$ is continuous on $\Kpc\times\partial\disk$, and possible values of $\Lambda_k$ on $\partial\disk$ are discrete, it follows that $\Lambda$ is constant on $\Kpc\times\partial\disk$.
In particular, $\Lambda|_{\Kpc\times\partial\disk} = \Lambda_{k_0}|_{\Kpc\times\partial\disk}=0$.
Then by (d) $A_{t,k}$ is fixed on $\partial\disk$ for all $(k,t)\in I\times\Kpc$.
\end{proof}

\subsection*{Proof of Proposition~\ref{pr:deform_in_EApl}.}
We will find $a,b\in(0,1)$ and a shift function $\Lambda:\Ksp\times(\disko)\to\RRR$ for $\Omega$ such that the corresponding homotopy $\Ahom_t$ constructed in Lemma~\ref{lm:prop_homot_A} will satisfy~\eqref{equ:Ahom_01} and~\eqref{equ:omegat_L_Dapl__K_EAd}.

\myemph{Choice of $\Lambda$.}
Let $\Lambda':\Ksp\times(\disko)\to\RRR$ be any shift function for $\Omega$.
Since $\Omega_{k_0}$ is fixed on $\partial\disk$ for some $k_0\in\Kpc$, we have that $\Lambda'_{k_0}|_{\partial\disk}=n\theta(\partial\disk)$ for some $n\in\ZZZ$.

Define another function $\Lambda:\Ksp\times(\disko)\to\RRR$ by 
$$\Lambda(k,z) \;=\; \Lambda'(k,z)\;-\;n\,\theta(\partial\disk).$$
Then $\Lambda$ is also a shift function for $\Omega$ in the sense of~\eqref{equ:Omega_ShA_Lambda} and satisfies 
\begin{equation}\label{equ:Lambda_pD2_0}
\Lambda_{k_0}|_{\partial\disk}=0.
\end{equation}

\myemph{Choice of $a,b\in(0,1)$.}
Notice that $\Omega_{k}(\Uman_{b})=\Uman_{b}$ for all $k\in\Ksp$ and $b\in(0,1]$.
We claim that \myemph{there exists $b\in (0,1)$ such that the map $\Omega_{k}:\Uman_{b}\to\Uman_{b}$ is a diffeomorphism for all $k\in I$.}
Indeed, by definition of $\EApl$ the map $\Omega_{k}$ is a diffeomorphism at $\orig$ for each $k\in\Ksp$.
This implies existence of $b$ in the case (i), i.e. when $\Ksp$ is a point.
In the case (ii) the assumption $r\geq1$ means that the partial derivatives of $\Omega_{k}$ are continuous functions on $\Ksp\times\disk$.
Then existence of $b$ follows from compactness of $\Ksp\times\disk$.

\medskip

Take arbitrary $a\in(0,b)$ and let $\Ahom_t$ be a homotopy constructed in Lemma~\ref{lm:prop_homot_A} for $\Lambda$ and $a,b$.
We claim that $\Ahom$ satisfies~\eqref{equ:Ahom_01} and~\eqref{equ:omegat_L_Dapl__K_EAd}.

By (a) of Lemma~\ref{lm:prop_homot_A} $\Ahom_0=\omega$.

Let us prove that $\Ahom_1(\Ksp)\subset\DAd$, i.e. for each $k\in\Ksp$ the map $\Ahom_{1,k}=\Omega'_{k}$ is a diffeomorphism of $\disk$ fixed on $\partial\disk$.
By (b) of Lemma~\ref{lm:prop_homot_A} $\Omega'_{k}$ is fixed even on $\disk\setminus\Uman_{b}$.
Moreover by assumption on $b$ we have that $\Omega_{k}:\Uman_{b}\to\Uman_{b}$ is a diffeomorphism, whence by (c) of Lemma~\ref{lm:prop_homot_A} $A_{t,k}=\Ahom_t(k)$ is also a self-diffeomorphism of $\Uman_{b}$ and therefore of all $\disk$.

By assumption $\Omega_{k}:\disk\to\disk$ is a diffeomorphism for all $k\in\Ldf$.
Then again by (c) of Lemma~\ref{lm:prop_homot_A} $A_{t,k}=\Ahom_t(k)$ is also a self-diffeomorphism $\disk$ for all $k\in\Ldf$.
In other words $\Ahom_t(\Ldf)\subset \DApl$.

Finally, the inclusion $\Ahom_t(\Kpc)\subset\EAd$ follows from~\eqref{equ:Lambda_pD2_0} and (e) of Lemma~\ref{lm:prop_homot_A}.

\section{Proof of Theorem~\ref{th:DA_EA_weak_he}}\label{sect:proof:th:DA_EA_weak_he}
First we prove (1) and (2) for the inclusions $\DApl\subset\EApl$ and $\DAd\subset\EAd$.
Then we establish (3) and deduce from it (1) and (2) for the inclusion $\DA\subset\EA$.

(1) We have to show that $\pi_n(\EE^{r},\DD^{r})=0$ for all $n\geq0$ if $r\geq1$
Then the result will follow from exact homotopy sequence of the pair $(\EE^{r},\DD^{r})$.

Let $\omega:(I^n,\partial I^n)\to(\EE^{r},\DD^{r})$ 
be a continuous map representing some element of the relative homotopy set $\pi_n(\EE^{r},\DD^{r})$.
Our aim is to show that $\omega$ is homotopic as a map of pairs to a map into $\DD$, i.e. $\omega=0$ in $\pi_n(\EE^{r},\DD^{r})$, whence we will get $\pi_n(\EE^{r},\DD^{r})=0$.

{\em\myemph{Inclusion $\DApl\subset\EApl$.}}
If $r\geq1$, then applying Proposition~\ref{pr:deform_in_EApl} to the case $\Ksp=I^n$, $\Ldf=\partial I^n$ we obtain that $\omega$ is homotopic as a map of pairs into $\DApl$.

{\em\myemph{Inclusion $\DAd\subset\EAd$.}}
Since $$(\EAd,\DAd)\ \subset\ (\EApl,\DApl),$$ we see that $\omega$ is also an element of $\pi_n(\EAplr{r},\DAplr{r})$, which as just shown is trivial.
Then Proposition~\ref{pr:deform_in_EApl} can be applied to the case $\Ksp=\Kpc=I^n$ and $\Ldf=\partial I^n$, and we obtain that $\omega$ is homotopic as a map of pairs $(K,L)\to(\EAdr{r},\DAdr{r})$ to a map into $\DAdr{r}$.
Hence $\omega=0$ in $\pi_n(\EAdr{r},\DAdr{r})$.

\medskip

(2)
We have to show that the map $\pi_0\DD^{0}\to\pi_0\EE^{0}$ is surjective for all $r\geq0$.
Let $\hdif\in\EApl$.
It can be regarded as a map from the set $\Ksp$ consisting of a unique point into $\EA$:
$$
\omega:\Ksp\to\EE, \qquad
\omega(\Ksp)=\hdif.
$$
Then applying Proposition~\ref{pr:deform_in_EApl} we obtain that $\omega$ is $\Cinf$-homotopic to a map into $\DAd$,
whence the inclusion $\DAd\subset\EApl$ yields a surjecitve map $\pi_0\DAdr{r}\to\pi_0\EAplr{r}$ for all $r\geq0$.
Therefore in the following diagram induced by inclusions all arrows are surjective:
\begin{equation}\label{equ:pi0_epimorphisms}
\begin{CD}
\pi_0\DAdr{r} @>>> \pi_0\EAdr{r} \\
@VVV @VVV \\
\pi_0\DAplr{r} @>>> \pi_0\EAplr{r}
\end{CD}
\end{equation}
The proof of the surjectivity $\pi_0\DAr{r}\to\pi_0\EAr{r}$ is the same as in (1).

\medskip 

(3) 
It is well known and is easy to prove that \myemph{for a topological semigroup $\EE$ the set $\pi_0\EE$ of path components of $\EE$ admits a natural semigroup structure such that the natural projections $\EE\to \pi_0\EE$ is a semigroup homomorphism.
If $\EE$ is a group, then so is $\pi_0\EE$.}

\myemph{If $\DD\subset\EE$ is a subsemigroup, then the induced map $\pi_0\DD\to\pi_0\EE$ is a semigroup homomorphism.}

In our case $\EE^r$ is a topological semigroup and $\DD^r$ is a topological group.
From (1) we get that for $r\geq1$ the homomorphism $\pi_0\DD\to\pi_0\EE$ is a bijection, whence it is a semigroup isomorphism.
But $\pi_0\DD$ is a group, whence so is $\pi_0\EE$.

Let us prove that \myemph{all path components of $\EE$ are homeomorphic each other}.
By (1) and (2) $i_{0}:\pi_0\DD\to\pi_0\EE$ is surjective for any $\Wr{r}$ topologies, $r\geq0$.
In particular, this implies that each path component of $\EE$ contains an invertible element.
Now the result is implied by the following statement:
\begin{claim}
Let $\EE$ be a topological semigroup such that each path component of $\EE$ contains an invertible element.
Then all path components of $\EE$ are homeomorphic each other.

Moreover, let $\DD$ be the subgroup consisting of all invertible elements.
Then for any two path components $\EE_1$ and $\EE_2$ there exists a homeomorphism $Q:\EE_1\to\EE_2$ such that $Q(\EE_1\cap\DD) =\EE_2\cap\DD$.
\end{claim}
\begin{proof}
Let $\hdif_1\in\EE_1$ and $\hdif_2\in\EE_2$ be any invertible elements.
Then we can define a homeomorphism $Q:\EE_1\to\EE_2$ by $Q(\hdif) = \hdif_2\cdot \hdif_1^{-1} \cdot \hdif$.
Evidently it is continuous, its inverse is given by $Q^{-1}(\gdif) = \hdif_1\cdot \hdif_2^{-1} \cdot \gdif$, and $Q(\EE_1\cap\DD) =\EE_2\cap\DD$.
\end{proof}

\medskip 

(1) and (2) for the inclusion $\DA\subset\EA$.
Put $$\DD'=\DA\setminus\DApl, \qquad \EE'=\EA\setminus\EApl.$$
Then $\EE'$ consists of full path components of $\EA$ with respect each of $\Wr{r}$ topologies.
Hence we have to prove our statement for the inclusion $\DD'\subset\EE'$.
We can also assume that $\EE'\not=\varnothing$.

By Lemma~\ref{lm:change_to_diff} there exists $\gdif\in\DD'$.
Hence we can define the following map $Q:\EE'\to\EApl$ by $Q(\hdif) = \gdif^{-1}\circ \hdif$ for $\hdif\in\EE'$.
Evidently, $Q$ is a homeomorphism onto with respect to any of $\Wr{r}$ topologies.
Moreover, $Q(\DD')=\DApl$.
Hence $\pi_n(\EE',\DD') = \pi_n(\EApl,\DApl)$.
It remains to note that by $\pi_n(\EApl,\DApl)=0$ if either $r\geq1$ and $n\geq0$, or $r=0$ and $n=0$.

\section{Proof of Theorem~\ref{th:EAd_EApl}}\label{sect:proof:th:EAd_EApl}
The proof is similar to the one given in \S\ref{sect:proof:th:DA_EA_weak_he}.
Let $$\omega:(I^n,\partial I^n)\to(\EAplr{r},\EAdr{r})$$ be a continuous map being a representative of some element in the relative homotopy set $\pi_n(\EAplr{r},\EAdr{r})$.

We have to show that $\omega$ is homotopic as a map of pairs to a map into $\EAd$.
Again we will apply Proposition~\ref{pr:deform_in_EApl} but now the situation is more complicated.

For $n\not=1$ denote $\Ksp=I^n$ and $\Kpc=\partial I^n$.
Then $\Kpc$ is path connected and by Proposition~\ref{pr:deform_in_EApl} $\omega$ is homotopic as map of pairs $$(\Ksp,\Kpc)\to (\EAplr{r},\EAdr{r})$$ to a map into $\EAd$.
In this case we can take $b$ arbitrary, and therefore the arguments hold for the case $r=0$ as well.
This implies $\pi_n(\EAplr{r},\EAdr{r})=0$ for all $n\not=1$ and $r\geq0$.

Suppose $n=1$.
Then $I^1=[0,1]$ and $\partial I^1= \{0,1\}$ is not connected, so Proposition~\ref{pr:deform_in_EApl} can be applied only to each of path components $\{0\}$ and $\{1\}$ of $\partial I^1$, and this is the reason why
\begin{equation}\label{equ:pi1_EApl_EAd_Z}
\pi_1(\EAplr{r},\EAdr{r},\id_{\disk})\approx \ZZZ.
\end{equation}
To prove~\eqref{equ:pi1_EApl_EAd_Z} use $0\in I^1$ and $\id_{\disk}\in\EApl$ as base points, and thus assume that $\omega(0)=\id_{\disk}$.
Consider the $(I^1,r)$-deformation in $\EAplr{r}$ corresponding to $\omega$:
$$\Omega:I^1\times \disk\to\disk, \qquad \Omega(k,z)=\omega(k)(z).$$
Then $\Omega_0=\id_{\disk}$ and therefore the zero function $\Lambda_0=0$ is a shift function for $\Omega$.
By Lemma~\ref{lm:ext_shift_func_via_homotopy} $\Lambda_0$ extends to a \myemph{unique} $(I^1,r)$-deformation 
$$
\Lambda:I^1\times(\disko)\to\RRR
$$
being a shift function for $\Omega$ on $\disko$ in the sense of that lemma.
In particular the last function $\Lambda_1$ is a shift function for $\Omega_1\in\EAd$ which is fixed on $\partial\disk$.
Then by Proposition~\ref{pr:deform_in_EApl} applied to $\Kpc=\{1\}$, we get that $\Lambda_1$ takes constant value on $\partial\disk$ being an integer multiple of the period of orbit $\partial\disk$ with respect to $\AFld$.
Thus
$$
\Lambda(\partial\disk) = \rho_{\omega} \cdot \theta(\partial\disk)
$$
for some $\rho_{\omega}\in \ZZZ$.
Evidently, $\rho_{\omega}$ counts the number of ``full rotations'' of $\partial\disk$ during the homotopy $\Omega$.
We claim that the correspondence $\rho:\omega \mapsto \rho_{\omega}$ yields an isomorphism~\eqref{equ:pi1_EApl_EAd_Z}.

It is easy to see that $\rho$ induces a \myemph{surjective homomorphism}
$$
  R:\pi_1(\EAplr{r},\EAdr{r})\to\ZZZ.
$$

To show that $R$ is a monomorphism suppose that $\rho(\omega)=0$, so $\Lambda_0(\partial\disk)=\Lambda_1(\partial\disk)=0$.
Then it follows from (d) of Lemma~\ref{lm:prop_homot_A} that
$$A(t,0,z)=A(t,1,z)=z$$ for all $t\in I$ and $z\in\partial\disk$.
In other words, $\Ahom_t(0),\Ahom_t(1)\in\EAdr{r}$ for all $t\in I$.
Thus $\omega$ is homotopic to a map into $\EAd$ via a homotopy relatively $\partial I^1$, and therefore is represents a trivial element of $\pi_1(\EAplr{r},\EAdr{r})$.
This implies that $R$ is an isomorphism.

\section{Proof Theorem~\ref{th:EA_non-deg-cond}}\label{sect:proof:th:EA_non-deg-cond}
Let $\AFld$ be a \TC\ vector field on $\disk$ such that $\nabla\AFld$ is degenerate and $\ker\jet\subset\imShA$.

\medskip 

\NFZ\
If $\nabla\AFld=0$, then the relation $\imShA=\ker\jet$ follows from the assumption $\imShA\supset\ker\jet$ and \NFZ\ of Proposition~\ref{pr:info_on_TC}.

Suppose that $\{\id\}$ is the path component of $\id$ in $\LA$.
Since $\jet$ is continuous from $\Wr{r}$ topology of $\EA$ for $r\geq1$, it follows that $\EidAr{r} \subset  \ker\jet =\imShA \subset \EidAr{r}$.
This also implies~\eqref{equ:pi0_EApl_EA}.

\medskip 

\NFS\
Suppose that $\nabla\AFld =\left(\begin{smallmatrix} 0 & a \\ 0 & 0  \end{smallmatrix}  \right)$ for some $a\not=0$.
We have to show that $\imShA =  \EidAr{1}  = \jet^{-1}(\Apl)$.

It follows from the definition, see~\eqref{equ:Apl_Amn}, that $\Ag$ is a group, $\Apl$ is its unity component in $\GLRpl{2}$, and $\Amn$, $\Aopl$, $\Aomn$ are another path components of $\Ag$.
Since $\jet$ is continuous in $\Wr{r}$ topology of $\EApl$ for $r\geq 1$, it follows that inverse images of these path components are open-closed in $\EA$.
On the other hand $\imShA$ is path connected in all $\Wr{r}$-topologies, as a continuous image of path connected space $\Ci{\Mman}{\RRR}$, whence 
$$\imShA \ \subset \ \EidAr{1} \ \subset \ \jet^{-1}(\Apl).$$

Conversely, let $\hdif\in\jet^{-1}(\Apl)$, so 
$\jt{\hdif} = \left(\begin{smallmatrix}1 & a\tau \\ 0 & 1 \end{smallmatrix} \right)$ for some $\tau \in \RRR$.
We have to show that $\hdif\in\imShA$.

Evidently $\jt{\hdif}$ coincides with 
$$\exp(\tau\cdot \nabla\AFld) = \exp{\left(\begin{smallmatrix}0 & a\tau \\ 0 & 0 \end{smallmatrix} \right)}.$$
Consider the flow $(\AFlow_{t})$ of $\AFld$.
Then, $\jt{\AFlow_{t}}=\exp{\left(\begin{smallmatrix}0 & at \\ 0 & 0 \end{smallmatrix} \right)}$ for all $t\in\RRR$.
Hence $\jt{\AFlow_{\tau}}=\jt{\hdif}$.

Define the map $\gdif:\disk\to\disk$ by $\gdif(z) = \AFlow(\hdif(z),-\tau)=\AFlow_{-\tau}\circ\hdif$.
Then $\jt{\gdif}= \left(\begin{smallmatrix}1 & 0 \\ 0 & 1 \end{smallmatrix} \right)$, i.e. $\gdif\in\ker\jet \subset\imShA.$

In other words, $\gdif(z) = \AFlow(z,\afunc(z))$ for some $\afunc\in\Ci{\disk}{\RRR}$.
Put $\bfunc(z)= \afunc(z)+\tau$.
Then $\hdif(z) = \AFlow(z,\bfunc(z))$, i.e. $\hdif\in\imShA$.

\medskip

(3) 
Due to~\eqref{equ:pin_EAd_EApl} we have only to show that the mapping $$i_1:\pi_1\EAdr{r} \to \pi_1\EAplr{r}$$ induced by the inclusion is an isomorphism.
Moreover, by exactness of the sequence~\eqref{equ:exseq_EAd_EApl} it remains to show that $i_1$ is surjective.

Let $\omega:I\to\EAplr{r}$ be a continuous map representing a loop in $\EAplr{r}$, i.e. 
\begin{equation}\label{equ:omega_dI_id}
\omega(0)=\omega(1)=\id_{\disk}.
\end{equation}
We have to show that $\omega$ is $r$-homotopic relatively $\partial I$ to a map into $\EAdr{r}$.

It follows from~\eqref{equ:omega_dI_id} that $\omega(I)$ included into $\EidAplr{r}$ which by (1) and (2) coincides with $\imShA$.
Thus $\omega(I)\subset\imShA$.
Moreover, $\omega(\partial I)\subset \ker\jet$.
Then by Lemma~\ref{lm:deform_in_imShA} $\omega$ is homotopic to a map into $\ker\jet$ relatively $\partial I$.
Hence we can assume that $\omega$ is a loop in $\ker\jet$.

Consider the $(I,r)$-deformation corresponding to $\omega$:
$$
\Omega:I\times \disk\to\disk, \qquad \Omega(t,z)=\omega(t)(z).
$$
Then $\Omega_0=\Omega_1=\id_{\disk}$ and $\Omega_{k}\in\ker\jet\subset\imShA$ for all $k\in I$.

In particular, every $\Omega_k$ has a $\Cinf$ shift function $\Lambda_k:\disk\to\RRR$ defined on all of $\disk$.
Since $\nabla\AFld$ is degenerate, we have by Lemma~\ref{lm:exist_shift_func} that such $\Lambda_k$ is unique.
In particular, $\Lambda_0=\Lambda_1=0$.

Then it follows from Lemma~\ref{lm:ext_shift_func} that the map $\Lambda:I\times(\disko)\to\RRR$ defined by $\Lambda(k,z)=\Lambda_k(z)$ is a $(I,r)$-deformation being a shift function for $\Omega$.

Take any $a,b\in(0,1)$ such that $a<b$ and consider the homotopy $\Ahom_t$ of $\omega$ into $\EAd$ defined by~\eqref{equ:homotopy_to_diffeo}.
Since $\Omega(0,z)=\Omega(1,z)=z$ for all $z\in\disk$, we obtain from (d) of Proposition~\ref{pr:deform_in_EApl} $\Ahom_t(0,z)=\Ahom_t(1,z)=z$ for all $t\in I$.
In other words $\Ahom_t$ is a homotopy relatively $\partial I$.

\medskip

(4) This statement follows from \NFZ-(3) and the well known fact that any finite subgroup of $\GLRpl{2}$ is cyclic.

Theorem~\ref{th:EA_non-deg-cond} is completed.

\end{document}